\documentclass[12pt]{amsart}
\usepackage{amssymb,amsmath,amsthm,fullpage,color,soul}
\usepackage{pdfsync}

\newtheorem{thm}{Theorem}[section]
\newtheorem{prop}[thm]{Proposition}
\newtheorem{lem}[thm]{Lemma}

\theoremstyle{definition}
\newtheorem{defn}[thm]{Definition}

\theoremstyle{remark}
\newtheorem{rem}[thm]{Remark}
\newtheorem{remark}[thm]{Remark}
\numberwithin{equation}{section}

\newcommand{\abs}[1]{\left\vert#1\right\vert}

\newcommand{\ddbar}{\partial\bar\partial}
\newcommand{\dom}[1]{\mathrm{Dom}(#1)}

\DeclareMathOperator{\re}{\mathrm{Re}}
\DeclareMathOperator{\im}{\mathrm{Im}}
\renewcommand{\L}{\mathcal{L}}
\DeclareMathOperator{\dist}{dist}

\DeclareMathOperator{\rea}{Reach}

\newcommand{\LL}{\bar L}

\newcommand{\R}{\mathbb R}

\newcommand{\N}{\mathbb N}
\newcommand{\C}{\mathbb C}

\DeclareMathOperator{\Rre}{Re}
\DeclareMathOperator{\Dom}{Dom}

\newcommand{\bd}{\textrm{b}}

\newcommand{\p}{\partial}

\newcommand{\z}{\bar z}

\newcommand{\dbar}{\bar\partial}
\newcommand{\dbars}{\bar\partial^*}

\newcommand{\vp}{\varphi}

\newcommand{\atopp}[2]{\genfrac{}{}{0pt}{2}{#1}{#2}}

\newcommand{\ep}{\epsilon}
\newcommand{\I}{\mathcal{I}}

\DeclareMathOperator{\Tr}{Tr}
\newcommand{\bfem}[1]{\textbf{\emph{#1}}}
\newcommand{\la}{\langle}
\newcommand{\ra}{\rangle}

\begin{document}

\title{Closed range of $\dbar$ on unbounded domains in $\C^n$}

\author{Phillip S. Harrington and Andrew Raich}

\thanks{The second author was partially supported by NSF grant DMS-1405100.}

\address{Department of Mathematical Sciences, SCEN 309, 1 University of Arkansas, Fayetteville, AR 72701}
\email{psharrin@uark.edu \\ araich@uark.edu}



\begin{abstract}In this article, we establish a general sufficient condition for closed range of the Cauchy-Riemann operator $\bar\partial$ in appropriately weighted $L^2$ spaces on $(0,q)$-forms for a fixed $q$ on domains in $\mathbb{C}^n$.
The domains we consider may be neither bounded nor pseudoconvex, and our condition is a generalization of the classical $Z(q)$ condition that we call weak $Z(q)$. We provide examples that explain the
necessity of working in weighted spaces for closed range in $L^2$.
\end{abstract}

\maketitle

\section{Introduction}
\label{sec:introduction}

Suppose that $\Omega\subset\mathbb{C}^n$ is a smooth domain that may be neither bounded nor pseudoconvex.  Our goal in this paper is to study sufficient conditions for closed range of the Cauchy-Riemann operator in weighted $L^2$ spaces.  When $\Omega$ is bounded and pseudoconvex, this follows from the classic results of H\"ormander \cite{Hor65}.  Recent work of Herbig and McNeal \cite{HeMc16} studies sufficient conditions for closed range on unbounded pseudoconvex domains in unweighted spaces.

In \cite{HaRa15}, the authors introduced a condition known as weak $Z(q)$ which implies that the Cauchy-Riemann operator has closed range in $L^2_{0,q}$ or $L^2_{0,q+1}$ on bounded domains.  This condition is built on the authors' earlier work in \cite{HaRa11}, and is inspired by related conditions in \cite{Ho91}, \cite{AhBaZa06}, and \cite{Zam08}, as well as the classic $Z(q)$ condition (see \cite{Hor65}, \cite{FoKo72}, \cite{AnGr62}, or \cite{ChSh01}).  We will review our definition of weak $Z(q)$ in Section \ref{sec:weakly_z(q)_domains}, but for now we recall that the special case of $Z(q)$ is the case where the Levi-form has either $q+1$ negative or $n-q$ positive eigenvalues at every boundary point.  On bounded domains, there must be at least one strictly pseudoconvex boundary point, so by continuity a bounded $Z(q)$ domain in $\mathbb{C}^n$ must have at least $n-q$ positive eigenvalues at every boundary point.  Hence, a large class of interesting local examples (those with $q+1$ negative eigenvalues), cannot be realized globally as bounded domains in $\mathbb{C}^n$ (or indeed any Stein manifold).

Such examples might exist when considering unbounded domains.  However, a simple counterexample demonstrates the critical role played by the weight function on such domains.  Suppose there exists a constant $C>0$ such that for every $u\in L^2_{0,q}(\Omega)\cap\Dom(\dbar)\cap\Dom(\dbars)$, we have the closed range estimate
\[
  \|u\|\leq C(\|\dbar u\|+\|\dbars u\|),
\]
where $\dbars$ is the $L^2$ adjoint of $\dbar$ (see Section \ref{sec:weakly_z(q)_domains} for details on the notation).  Suppose that for every $R>0$, there exists $z_R\in\Omega$ such that $B(z_R,R)\subset\Omega$.  This is possible even on strictly pseudoconvex unbounded domains such as the half-space bounded by the Heisenberg group: $\Omega=\{z\in\mathbb{C}^n:\im z_n>|z_1|^2+\cdots+|z_{n-1}|^2\}$.  Let $u_1\in C^\infty_{0,(0,q)}(B(0,1))$ be nontrivial, and define $u_R(z)=\frac{1}{R^n}u_1\left(\frac{z-z_R}{R}\right)$.  Then $u_R\in C^\infty_{0,(0,q)}(B(z_R,R))\subset C^\infty_{0,(0,q)}(\Omega)$.  By assumption,
\[
  \|u_1\|=\|u_R\|\leq C(\|\dbar u_R\|+\|\dbars u_R\|)=R^{-1}C(\|\dbar u_1\|+\|\dbars u_1\|).
\]
Since this must hold for every $R>0$, we have a contradiction.  Thus, closed range estimates in $L^2$ are impossible on many unbounded domains, so we must consider weighted $L^2$ spaces to obtain closed range estimates.

With the tools of \cite{HaRa15} in place, the present paper establishes sufficient conditions for closed range of the Cauchy-Riemann complex on a large class of unbounded domains.  We introduce our key definitions to classify the domains under consideration in Section \ref{sec:weakly_z(q)_domains}.  Section \ref{sec:basic estimate, proof of L^2 case} will outline the key computations from \cite{HaRa15} to prove closed range in weighted $L^2$ spaces.  In Section \ref{sec:closed_range_in_unweighted_spaces} we briefly outline the applications of our techniques to closed range in unweighted $L^2$ spaces, as a special case of the results in \cite{HeMc16}.  We conclude with the construction of examples satisfying our hypotheses in Section \ref{sec:examples}.

\section{Weakly $Z(q)$ domains.}
\label{sec:weakly_z(q)_domains}

\subsection{Notation}
Let $\Omega\subset \C^n$ be a domain with $C^m$ boundary $\bd\Omega$.
\begin{defn}\label{defn:uniform_defining_function}We say that a defining function $\rho$ for $\Omega$ is \bfem{uniformly $C^m$} if there exists an open neighborhood $U$ of $\bd\Omega$ such that $\dist(\bd\Omega,\bd U)>0$, $\|\rho\|_{C^m(U)}<\infty$, and $\inf_U |\nabla\rho|>0$.
\end{defn}
This is trivial on domains with compact boundary, but on unbounded domains we provided counterexamples, a large class of examples, and a complete characterization in terms of the signed distance function in \cite{HaRa13}.

We identify real $(1,1)$-forms with a hermitian matrix as follows:
\[
c=\sum_{j,k=1}^{n} i c_{j\bar k}\, dz_j\wedge d\z_k
\]

For a function $\alpha$, we denote $\alpha_k = \frac{\p\alpha}{\p z_k}$ and $\alpha_{\bar j} = \frac{\p\alpha}{\p\z_j}$.

Let $\rho:\C^n\to\R$ be a uniformly $C^m$-defining function for $\Omega$.
We denote the $L^2$-inner product on $L^2(\Omega,e^{-t|z|^2})$ by
\[
(f,g)_t = \int_\Omega \la f, g\ra\, e^{-t|z|^2} dV = \int_\Omega f \bar g\, e^{-t|z|^2} dV
\]
where $\la \cdot,\cdot \ra$ is the standard pointwise inner product on $\C^n$ and $dV$ is Lebesgue measure on $\C^n$. We denote
the induced surface area measure on $\bd\Omega$ by $d\sigma$. Also $\| f \|_t^2 = \int_\Omega |f|^2 e^{-t|z|^2}\, dV$.

Let $\I_q = \{ (i_1,\dots,i_q)\in \N^n : 1 \leq i_1 < \cdots < i_q\leq n \}$. For $I\in\I_{q-1}$, $J\in\I_q$, and $1\leq j \leq n$, let
$\ep^{jI}_{J} = (-1)^{|\sigma|}$ if $\{j\} \cup I = J$ as sets and $|\sigma|$ is the length of the permutation that takes $\{j\}\cup I$ to $J$. Set
$\ep^{jI}_J=0$ otherwise. We use the standard notation that if $u = \sum_{J\in\I_q} u_J\, d\z_J$, then
\[
u_{jI} = \sum_{J\in\I_q} \ep^{jI}_J u_J.
\]

Let $L^{t}_j = \frac{\p}{\p z_j} - t\z_j = e^{t|z|^2}\frac{\p}{\p z_j} e^{-t|z|^2}$ and let $\dbars_t:L^2_{0,q+1}(\Omega,e^{-t|z|^2})\rightarrow L^2_{0,q}(\Omega,e^{-t|z|^2})$ be the $L^2$-adjoint of
$\dbar:L^2_{0,q}(\Omega, e^{-t|z|^2}) \to L^2_{0,q+1}(\Omega, e^{-t|z|^2})$. This means that if
$f = \sum_{J\in\I_q} f_J\, d\z_J$ and $g = \sum_{K\in\I_{q+1}}g_K\, d\z_K\in\Dom(\dbars_t)$, then
\[
\dbar f = \sum_{\atopp{J\in\I_q}{K\in\I_{q+1}}}\sum_{k=1}^n \ep^{kJ}_K \frac{\p f_J}{\p\z_k}\, d\z_K
\qquad\text{and}\qquad
\dbars_t g = -\sum_{J\in\I_{q}} \sum_{j=1}^n L^t_j g_{jJ}\, d\z_J.
\]

The induced CR-structure on $\bd\Omega$  at $z\in\bd\Omega$ is
\[
T^{1,0}_z(\bd\Omega)  = \{ L\in T^{1,0}(\C) : \p\rho(L)=0 \}.
\]
Let $T^{1,0}(\bd\Omega)$ be the space of $C^{m-1}$ sections of $T^{1,0}_z(\bd\Omega)$ and $T^{0,1}(\bd\Omega) = \overline{T^{1,0}(\bd\Omega)}$.
We denote the exterior algebra generated by these spaces by $T^{p,q}(\bd\Omega)$. If $U$ is a suitably small neighborhood of $\bd\Omega$, we
use $\tau$ to denote the orthogonal projection and
restriction
\[
\tau : \Lambda^{p,q}(U)\to \Lambda^{p,q}(\bd\Omega).
\]

If we normalize $\rho$ so that $|d\rho|=1$ on $\bd\Omega$, then the Levi form $\L$ is the real element of $\Lambda^{1,1}(\bd\Omega)$ defined by
\[
\L(-i L\wedge \LL) = i\p\dbar\rho(-iL\wedge\LL)
\]
for any $L\in T^{1,0}(\bd\Omega)$.

\begin{defn}\label{defn:tubular nbhd}Given a set $M\subset\C^n$, a \bfem{tubular neighborhood} of $M$ is an open set $U_r$ of the form
$U_{r} = \{p\in\C^n : \dist(p,M)<r\}$ where $\dist(\cdot,\cdot)$ is the Euclidean distance function. We call $r$ the \bfem{radius} of $U_r$.
\end{defn}

\subsection{Weak $Z(q)$ domains and closed range for $\dbar$}

The following definition was introduced in \cite{HaRa15}, building on ideas in \cite{HaRa11}.
\begin{defn}\label{defn:weak Z(q)}
Let $\Omega\subset \C^n$ be a domain with a uniformly $C^m$ defining function $\rho$, $m\geq 2$. We say $\bd\Omega$ (or $\Omega$) satisfies
\bfem{Z(q) weakly}
if there exists a hermitian matrix $\Upsilon=(\Upsilon^{\bar k j})$ of functions on $b\Omega$ that are uniformly bounded in $C^{m-1}$ such that $\sum_{j=1}^{n}\Upsilon^{\bar k j}\rho_j=0$ on $b\Omega$ and:
\begin{enumerate}\renewcommand{\labelenumi}{(\roman{enumi})}
 \item All eigenvalues of $\Upsilon$ lie in the interval $[0,1]$.

 \item $\mu_1+\cdots+\mu_q-\sum_{j,k=1}^n\Upsilon^{\bar k j}\rho_{j\bar k}\geq 0$ where  $\mu_1,\ldots,\mu_{n-1}$ are the eigenvalues of the Levi form $\L$ in increasing order.

 \item $ \inf_{z\in\bd\Omega} \{ |q-\Tr(\Upsilon)|\} >0$.

\end{enumerate}
\end{defn}
Weak $Z(q)$ is motivated by the basic identity (see \eqref{eq:dbar dbars from Gans} below).  The third term on the right-hand side in \eqref{eq:dbar dbars from Gans} may not be positive if the boundary is not pseudoconvex, so we wish to carry out additional integrations by parts in the first term on the right-hand side to create new boundary terms.  $\Upsilon$ dictates how much we integrate by parts in each direction.  Property (i) in Definition \ref{defn:weak Z(q)} guarantees that the gradient terms corresponding to the first term on the right-hand side of \eqref{eq:dbar dbars from Gans} remain positive.  Property (ii) guarantees that the boundary terms corresponding to the third term on the right-hand side of \eqref{eq:dbar dbars from Gans} also remain positive.  Property (iii) guarantees that the $L^2$ terms corresponding to the second term on the right-hand side of \eqref{eq:dbar dbars from Gans} have a coefficient that is bounded away from zero.  Proposition \ref{prop:basic identity} is the result of this integration by parts.

On bounded domains, the classical $Z(q)$ condition implies weak $Z(q)$.  In a neighborhood of each boundary point, we set $\Upsilon$ equal to the projection onto the span of the negative eigenspaces of the Levi-form.  Property (i) in Definition \ref{defn:weak Z(q)} must be satisfied since $\Upsilon$ is a projection.  For property (ii), we note that if the Levi-form has $p$ positive eigenvalues and $p\geq n-q$, then we have $\mu_1+\cdots+\mu_q-\sum_{j,k=1}^n\Upsilon^{\bar k j}\rho_{j\bar k}=\mu_{n-p}+\cdots+\mu_q>0$, while if $n-1-p\geq q+1$, then we have $\mu_1+\cdots+\mu_q-\sum_{j,k=1}^n\Upsilon^{\bar k j}\rho_{j\bar k}=-\mu_{q+1}-\cdots-\mu_{n-1-p}>0$.  Property (iii) follows since $Z(q)$ domains never have exactly $q$ negative eigenvalues.  These local constructions can be patched together to obtain a global $\Upsilon$.  On unbounded domains, this relationship is less clear, since uniform bounds on the derivatives of the Levi-form do not necessarily imply uniform bounds on the derivatives of the eigenvectors of the Levi-form (especially near points with repeated eigenvalues), but see \cite{HaRa16a} for a partial result.

Our main result is the following:
\begin{thm}\label{thm:closed range for dbar}Let $\Omega\subset \C^n$ be a domain with connected boundary that admits a uniformly $C^m$ defining function, $m\geq 3$, and
satisfies weak $Z(q)$ for some $1\leq q \leq n-1$.  There exists a $T>0$ so that if $q-\Tr(\Upsilon)>0$ and $t\geq T$ or  $q-\Tr(\Upsilon)<0$ and $t\leq -T$, then
\begin{enumerate}\renewcommand{\labelenumi}{(\roman{enumi})}
\item The operator $\dbar: L^2_{0,\tilde q}(\Omega, e^{-t|z|^2})\to L^2_{0,\tilde q+1}(\Omega, e^{-t|z|^2})$ has closed range
for $\tilde q = q-1$ or $q$;
\item The operator $\dbars_t: L^2_{0,\tilde q+1}(\Omega, e^{-t|z|^2})\to L^2_{0,\tilde q}(\Omega, e^{-t|z|^2})$ has closed range
for $\tilde q = q-1$ or $q$;
\item The weighted $\dbar$-Neumann Laplacian defined by $\Box_{q,t} = \dbar\dbars_t + \dbars_t\dbar$ has closed range on $L^2_{0,q}(\Omega, e^{-t|z|^2})$;
\item The weighted $\dbar$-Neumann operator operator $N_{q,t}$  is continuous on $L^2_{0,q}(\Omega, e^{-t|z|^2})$;
\item The operators
\[
  \dbars_t N_{q,t}:L^2_{0,q}(\Omega, e^{-t|z|^2})\to L^2_{0,q-1}(\Omega, e^{-t|z|^2})
\]
and
\[
(\dbar N_{q,t})^*_t : L^2_{0,q+1}(\Omega, e^{-t|z|^2})\to L^2_{0,q}(\Omega, e^{-t|z|^2}),
\]
are continuous;
\item The operators
\[
  \dbar N_{q,t}:L^2_{0,q}(\Omega, e^{-t|z|^2})\to L^2_{0,q+1}(\Omega, e^{-t|z|^2})
\]
and
\[
  (\dbars_t N_{q,t})^*_t : L^2_{0,q-1}(\Omega, e^{-t|z|^2})\to L^2_{0,q}(\Omega, e^{-t|z|^2}),
\]
are continuous;
\item If $\tilde{q}=q$ or $q+1$ and $\alpha\in L^2_{0,\tilde{q}}(\Omega, e^{-t|z|^2})$
so that  $\dbar \alpha =0$, then there exists $u\in L^2_{0,\tilde{q}-1}(\Omega, e^{-t|z|^2})$ so that
\[
\dbar u = \alpha.
\]
\item $\ker(\Box_{q,t}) = \{0\}$. 
\end{enumerate}
\end{thm}

\begin{remark}
The hypothesis on $U_{\bd\Omega}$ ensures that there exists an $r>0$ so that
if $p\in\bd\Omega$, then $B(p,r)\cap \bd\Omega$ is connected. This guarantees that we can move away from the boundary a uniform distance
without intersecting another piece of $\bd\Omega$.
\end{remark}

\begin{remark}
  When $q-\Tr(\Upsilon)<0$, the Levi-form of $\Omega$ must have at least $q+1$ nonpositive eigenvalues at every boundary point, so $\Omega$ must be  very large  (i.e., this Levi-signature can not be globally realized on a bounded domain).  This is also the case when our weight is very large because the exponent is positive, so the space $L^2(\Omega,e^{-t|z|^2})\cap\ker\dbar$ is probably very small.  Since we are dealing with $q\geq 1$, we at least know that $L^2_{(0,q)}(\Omega,e^{-t|z|^2})\cap\ker\dbar$ is nontrivial: simply choose a smooth, compactly supported $(0,q-1)$-form $\psi$ and consider $\dbar\psi$.
\end{remark}

\section{The basic estimate}\label{sec:basic estimate, proof of L^2 case}
In our proof of Theorem \ref{thm:closed range for dbar}, we will use the weight $\vp=t|z|^2$.  We would also like to consider applications of more general weight functions.
For a generic $C^2$ weight $\varphi$, the final (non-error term) in Proposition \ref{prop:basic identity} below would be
\[
  \sum_{I\in\I_{q-1}}\sum_{j,k=1}^n \big(\varphi_{j\bar k} f_{jI}, f_{kI}\big)_\varphi-\sum_{J\in\I_{q}}\sum_{j,k=1}^n \big(\varphi_{j\bar k}\Upsilon^{\bar k j} f_{J}, f_{J}\big)_\varphi
\]
To keep this term positive for all $f$, we would need to replace (iii) in Definition \ref{defn:weak Z(q)} with
\[
  \inf_{z\in b\Omega}\lambda_1+\cdots+\lambda_q-\sum_{j,k=1}^n\varphi_{j\bar k}\Upsilon^{\bar k j}>0,
\]
where $\lambda_1,\ldots,\lambda_n$ are the eigenvalues of $\varphi_{j\bar k}$ arranged in increasing order (see the definition of $q$-compatible functions in \cite{HaRa11}).  To avoid this technicality and keep the exposition along the lines of \cite{HaRa15}, we will restrict to weights which generate a multiple of the Euclidean K\"ahler form, i.e., $\ddbar\varphi=t\ddbar|z|^2$.  Thus, $\lambda_1=\cdots=\lambda_n=t$ and $\sum_{j,k=1}^n\varphi_{j\bar k}\Upsilon^{\bar k j}=t\Tr(\Upsilon)$.  We will see that this still allows us enough flexibility to consider some interesting special cases in Section \ref{sec:closed_range_in_unweighted_spaces}.

Recall that the signed distance function $\tilde\delta$ is defined by $\tilde\delta=\delta$ outside $\Omega$ and $\tilde\delta=-\delta$ inside $\Omega$.  As shown by Theorem 1.3 in \cite{HaRa13}, if any defining function $\rho$ satisfies the hypotheses of Theorem \ref{thm:closed range for dbar}, then the signed distance function will satisfy these hypotheses.  We also note that the signed distance function always satisfies $|\nabla\tilde\delta|=1$ where defined, so we can use $\tilde\delta_{j\bar k}$ to compute the normalized Levi-form.

We first show that  $\Upsilon$ defined on the boundary in Definition \ref{defn:weak Z(q)} can be extended to all of $\mathbb{C}^n$ in a uniform way.
\begin{lem}\label{lem:extension of Upsilon}
Suppose $\Omega$ has a connected boundary, a uniformly $C^m$ defining function for some $m\geq 2$, and satisfies weak $Z(q)$ for some $1\leq q\leq n-1$.
Let $\Upsilon$ be as in Definition \ref{defn:weak Z(q)}. There exists a hermitian matrix
$\tilde\Upsilon$ of functions on $\mathbb{C}^n$ that are uniformly bounded in $C^{m-1}$ satisfying
\begin{enumerate}\renewcommand{\labelenumi}{(\roman{enumi})}
\item All eigenvalues of $\tilde\Upsilon$ lie in the interval $[0,1]$.

\item $\tilde\Upsilon|_{b\Omega}=\Upsilon$, so that $\mu_1+\cdots+\mu_q- \sum_{j,k=1}^n\tilde\Upsilon^{\bar k j}\tilde\delta_{j\bar k}\geq 0$ on $\bd\Omega$ where $\mu_1,\ldots,\mu_{n-1}$ are the eigenvalues of the Levi form in increasing order.

\item $\inf_{z\in\bar\Omega} \{ |q-\Tr(\tilde\Upsilon)|\} >0$.

\item There exists $\ep>0$ so that on the neighborhood $U_\ep$ of $\bd\Omega$ we have
 \begin{equation}\label{eqn:tau rho_k sum=0}
 \sum_{j=1}^n \tilde\Upsilon^{\bar k j}\tilde\delta_j=0.
 \end{equation}
\end{enumerate}
\end{lem}
\begin{proof}
Since $b\Omega$ is uniformly $C^2$, it has positive reach (see Lemma 2.3 in \cite{HaRa13}), so for $0<\ep<\frac{1}{2}\rea(b\Omega)$ and $p\in U_{2\ep}$, the map $\pi : p \mapsto \bd\Omega$ mapping $p$ to $\pi(p)$ where
$\pi(p)$ is the unique point in $\bd\Omega$ satisfying $|p-\pi(p)| = \dist(p,\bd\Omega)$ is well-defined.  By Lemma 2.4 in \cite{HaRa13}, the signed distance function $\tilde\delta$ is uniformly $C^m$ on $U_{2\ep}$.  Since Theorem 4.8 (3) and (5) in \cite{Fed59} imply $\pi(p)=p-\tilde\delta(p)\nabla\tilde\delta(p)$ on $U_{2\ep}$, we conclude that $\pi_p$ is uniformly $C^{m-1}$ on $U_{2\ep}$.

Define a cutoff function $\psi\in C^\infty(\C^n,[0,1])$ so that $\psi\big|_{U_\ep}\equiv 1$ and $\psi\big|_{U_{2\ep}^c} \equiv 0$. If
$q-\Tr(\Upsilon)>0$ and $p\in\C^n$, define
\[
\tilde\Upsilon(p)
= \psi(p)\Upsilon(\pi(p)).
\]
If $\Tr(\Upsilon)-q>0$ and $p\in\mathbb{C}^n$, set
\[
\tilde\Upsilon(p)
= \psi(p)\Upsilon(\pi(p))
+ (1-\psi(p))I.
\]

To prove (\ref{eqn:tau rho_k sum=0}), observe that $\nabla\tilde\delta(\pi(p))=\nabla\tilde\delta(p)$ for $p\in U_{2\ep}$ by Theorem 4.8 (3) in \cite{Fed59} and a continuity argument.
\end{proof}

We will no longer distinguish between $\Upsilon$ and its extension $\tilde\Upsilon$.  We next prove a simple density result, adapting techniques that can be found in \cite{Gan10} and \cite{HeMc16}:

\begin{lem}
\label{lem:bounded_density}
  Let $\Omega\subset\mathbb{C}^n$ be a $C^m$ domain, $m\geq 2$, and let $f\in L^2_{0,q}(\Omega,e^{-\varphi})\cap\dom\dbar\cap\dom{\dbar^*_\varphi}$ for some $C^2$ function $\varphi$.  Then there exists a sequence of bounded $C^m$ domains $\{\Omega_j\}$ and functions $f_j\in C^{m-1}(\Omega)$ such that $\Omega_j\cap B(0,j+2)=\Omega\cap B(0,j+2)$, $f_j\equiv 0$ on $\Omega\backslash B(0,j+2)$, $f_j|_{\Omega_j}\in\dom{\dbar^*_\varphi}$, and
  \begin{multline*}
    \left\|\dbar f_j\right\|_{L^2(\Omega_j,e^{-\varphi})}+\left\|\dbar^*_\varphi f_j\right\|_{L^2(\Omega_j,e^{-\varphi})}+\left\|f_j\right\|_{L^2(\Omega_j,e^{-\varphi})}\\
    \rightarrow \left\|\dbar f\right\|_{L^2(\Omega,e^{-\varphi})}+\left\|\dbar^*_\varphi f\right\|_{L^2(\Omega,e^{-\varphi})}+\left\|f\right\|_{L^2(\Omega,e^{-\varphi})}
  \end{multline*}
\end{lem}

\begin{proof}
Let $\chi:\mathbb{R}\rightarrow\mathbb{R}$ be a smooth cutoff function satisfying $\chi(x)\equiv 0$ on $(-\infty,0]$ and $\chi(x)\equiv 1$ on $[1,\infty)$.
For $r>0$ and $z\in\Omega$, let $f_r(z)=\chi\left(\frac{(r+1)^2-|z|^2}{2r+1}\right)f(z)$.  We can easily check that $f_r\in\dom\dbar$ and
\[
  \dbar f_r(z)=\chi\left(\frac{(r+1)^2-|z|^2}{2r+1}\right)\dbar f(z)-\chi'\left(\frac{(r+1)^2-|z|^2}{2r+1}\right)\frac{\dbar|z|^2}{2r+1}\wedge f(z)
\]
almost everywhere.  Since $\chi'\left(\frac{(r+1)^2-|z|^2}{2r+1}\right)$ is supported in $\overline{B(0,r+1)}\backslash B(0,r)$, $\frac{\dbar|z|^2}{2r+1}$ is uniformly bounded on $B(0,r+1)$, and $\left\|f\right\|_{L^2(\Omega\backslash B(0,r),e^{-\varphi})}\rightarrow 0$ as $r\rightarrow\infty$, we have $\dbar f_r\rightarrow\dbar f$ in $L^2(\Omega,e^{-\varphi})$.  Next, we choose $g\in L^2_{0,q-1}(\Omega,e^{-\varphi})\cap\dom\dbar$, and check
\begin{multline*}
  \left(f_r,\dbar g\right)_{L^2(\Omega,e^{-\varphi})}\\
  =\left(f,\dbar\left(\chi\left(\frac{(r+1)^2-|\cdot|^2}{2r+1}\right)g\right)+\chi'\left(\frac{(r+1)^2-|\cdot|^2}{2r+1}\right)\frac{\dbar|\cdot|^2}{2r+1}\wedge g\right)_{L^2(\Omega,e^{-\varphi})},
\end{multline*}
so $|\left(f_r,\dbar g\right)_{L^2(\Omega,e^{-\varphi})}|\leq\left\|\dbar^*_\varphi f\right\|_{L^2(\Omega,e^{-\varphi})}\left\|g\right\|_{L^2(\Omega,e^{-\varphi})}+C\left\| f\right\|_{L^2(\Omega,e^{-\varphi})}\left\|g\right\|_{L^2(\Omega,e^{-\varphi})}$ for some constant $C$ independent of $r$.  Hence, $f_r\in\dom{\dbar^*_\varphi}$, and we can check that $\dbar^*_\varphi f_r\rightarrow \dbar^*_\varphi f$ as $r\rightarrow\infty$ in $L^2(\Omega,e^{-\varphi})$.

Let $\Omega_r$ be a bounded $C^m$ domain satisfying $\Omega_r\cap B(0,r+2)=\Omega\cap B(0,r+2)$.  By standard density results (e.g., Lemma 4.3.2 in \cite{ChSh01}), we can approximate $f_r$ on $\Omega_r$ by a $C^{m-1}$ form with the necessary properties.
\end{proof}

At this point, Proposition \ref{prop:basic estimate} below will follow directly by applying the proof of Proposition 3.1 (3) in \cite{HaRa15} to the form $f_j$ on the domain $\Omega_j$, and then taking limits.  We emphasize that $\Omega_j$ does not necessarily have weak $Z(q)$ boundary, but $f_j$ will vanish in a neighborhood of all boundary points where weak $Z(q)$ fails, so the proof will carry through without problems.  For the sake of clarity, we outline the key steps below.

Since we are assuming $\ddbar\vp = t\ddbar|z|^2$, we have
\[
\sum_{I\in\I_{q-1}} \sum_{j,k=1}^n \la \vp_{j\bar k} f_{jI}, f_{kI} \ra = qt \sum_{J\in I_q} |f_J|^2
\]
and the Morrey-Kohn-H\"ormander identity follows:
\begin{prop}\label{prop:dbar dbars from Gans}
Let $\Omega$ be a bounded domain of class $C^2$ and let $\rho$ be a defining function for $\Omega$ such that
$|\nabla\rho|=1$ on $\bd\Omega$. Then for any $f = \sum_{J\in\I_q} f_J d\bar z_J\in C^1(\overline\Omega)\cap\Dom(\dbars_\varphi)$,
\begin{equation}
\label{eq:dbar dbars from Gans}
\| \dbar f\|_\varphi^2 + \| \dbars_\varphi f \|_\varphi^2
= \sum_{J\in\I_q} \sum_{j=1}^n \Big\| \frac{\p f_J}{\p\z_j} \Big\|_\varphi^2 +qt \sum_{J\in I_q} \|f_J\|^2_\varphi
+\sum_{I\in\I_{q-1}} \sum_{j,k=1}^n \int_{\bd\Omega} \rho_{j\bar k} f_{jI}\overline{f_{kI}} e^{-\varphi}d\sigma.
\end{equation}
\end{prop}
This equality is well-suited for pseudoconvex domains but not a general weak $Z(q)$-domain.  For such domains, we need additional integrations by parts in the $\Big\| \frac{\p f_J}{\p\z_j} \Big\|_\varphi^2$ terms to obtain a useful estimate.  The form $\Upsilon$ can be thought of as a rule for integrating by parts in these terms, as we will now demonstrate.

Recall that $L_j^\varphi=e^\varphi\frac{\partial}{\partial z_j}e^{-\varphi}=\frac{\partial}{\partial z_j}-\varphi_j$. The fact that $\Upsilon$ consists entirely of tangential components is recorded in
\eqref{eqn:tau rho_k sum=0} which allows us to integrate by parts twice to obtain
\begin{multline*}
  \sum_{j,k=1}^{n}\left(\Upsilon^{\bar kj}\frac{\partial f_J}{\partial\bar z_k},\frac{\partial f_J}{\partial\bar z_j}\right)_\varphi
=\sum_{j,k=1}^{n}\left(\Upsilon^{\bar kj}L^\varphi_j  f_J,L^\varphi_k  f_J\right)_\varphi-\sum_{j,k=1}^{n}\left(\left(\frac{\partial}{\partial z_j}\Upsilon^{\bar kj}\right)\frac{\partial f_J}{\partial\bar z_k},\frac{\partial f_J}{\partial\bar z_j}\right)_\varphi\\
+\sum_{j,k=1}^{n}\left(\left(\frac{\partial}{\partial\bar z_k}\Upsilon^{\bar kj}\right)L^\varphi_j  f_J,f_J\right)_\varphi+t \big(\Tr(\Upsilon) f_J, f_J\big)_\varphi.
\end{multline*}
Here, we have used $\left[L_j^\varphi,\frac{\partial}{\partial\bar z_k}\right]=\varphi_{j\bar k}=t I_{jk}$ where $I$ is the identity matrix.  To work with the remaining first order terms, we use the decomposition $I=(I-\Upsilon)+\Upsilon$ to break apart each derivative and then use integration by parts as necessary so that the $I-\Upsilon$ component of each derivative is of type $(0,1)$ and the $\Upsilon$ component is of type $(1,0)$.  This leads to the identity
\begin{multline*}
  -\sum_{j,k=1}^{n}\left(\left(\frac{\partial}{\partial z_j}\Upsilon^{\bar kj}\right)\frac{\partial f_J}{\partial\bar z_k},\frac{\partial f_J}{\partial\bar z_j}\right)_\varphi
+\sum_{j,k=1}^{n}\left(\left(\frac{\partial}{\partial\bar z_k}\Upsilon^{\bar kj}\right)L^\varphi_j  f_J,f_J\right)_\varphi\\
=2\Rre\Bigg\{  \sum_{J\in \I_q}\sum_{j,k,\ell=1}^{n}\left(\frac{\partial\Upsilon^{\bar kj}}{\p\z_k}\Upsilon^{\bar j\ell}
  L^\varphi_\ell  f_J,f_J\right)_\varphi
  -\sum_{J\in \I_q}\sum_{j,k,\ell=1}^{n} 				
   \left(\frac{\partial\Upsilon^{\bar kj}}{\partial z_j}(I_{k \ell}-\Upsilon^{\bar\ell k})\frac{\partial f_J}{\partial\bar z_\ell},f_J\right)_\varphi \Bigg\}\\
   -\sum_{j,k=1}^n   \int_{b\Omega} \left\la \Upsilon^{\bar kj}\rho_{j\bar k}  f_J , f_J\right\ra e^{-\varphi} d\sigma+ O (\| f \|_\varphi^2).
\end{multline*}
Note that this integration by parts will introduce second derivatives of $\Upsilon$, so we will need to assume that these are uniformly bounded in order to absorb them in the error term.

Combining these calculations with \eqref{eq:dbar dbars from Gans}, we have the basic identity:
\begin{prop}\label{prop:basic identity} Let $\Omega$ be a bounded domain of class $C^3$ and let $\rho$ be a defining function for $\Omega$ such that
$|\nabla\rho|=1$ on $\bd\Omega$, and let $\Upsilon$ be an $n\times n$ hermitian matrix of $C^2$ functions such that $\sum_{j=1}^n\Upsilon^{\bar k j}\rho_j=0$ on $\bd\Omega$ for all $1\leq k\leq n$.  Suppose $\varphi$ satisfies $\ddbar\varphi=t\ddbar|z|^2$.  Then for any $f = \sum_{J\in\I_q} f_J d\bar z_J\in C^1(\overline\Omega)\cap\Dom(\dbars_\varphi)$,
\begin{align*}
&\| \dbar f\|_\varphi^2 + \| \dbars_\varphi f \|_\varphi^2
=  \sum_{J\in \I_q}\sum_{j,k=1}^{n}\left((I_{jk}-\Upsilon^{\bar kj})\frac{\partial f_J}{\partial\bar z_k},\frac{\partial f_J}{\partial\bar z_j}\right)_\varphi
+\sum_{J\in \I_q}\sum_{j,k=1}^{n}\left(\Upsilon^{\bar kj}L^\varphi_j  f_J,L^\varphi_k  f_J\right)_\varphi \\
&+\sum_{I\in\I_{q-1}} \sum_{j,k=1}^n \int_{\bd\Omega}\big\la \rho_{j\bar k} f_{jI}, f_{kI} \big\ra e^{-\varphi}d\sigma
- \sum_{J\in \I_q}\sum_{j,k=1}^n   \int_{b\Omega} \left\la \Upsilon^{\bar kj}\rho_{j\bar k}  f_J , f_J\right\ra e^{-\varphi} d\sigma \\
&+ 2\Rre\Bigg\{  \sum_{J\in \I_q}\sum_{j,k,\ell=1}^{n}\left(\frac{\partial\Upsilon^{\bar kj}}{\p\z_k}\Upsilon^{\bar j\ell}
  L^\varphi_\ell  f_J,f_J\right)_\varphi
  -\sum_{J\in \I_q}\sum_{j,k,\ell=1}^{n} 				
   \left(\frac{\partial\Upsilon^{\bar kj}}{\partial z_j}(I_{k \ell}-\Upsilon^{\bar\ell k})\frac{\partial f_J}{\partial\bar z_\ell},f_J\right)_\varphi \Bigg\}\\
&+ \sum_{J\in\I_q} t \big((q-\Tr(\Upsilon)) f_J, f_J\big)_\varphi + O (\| f \|_\varphi^2),
\end{align*}
where $O(\|f\|_{\varphi}^2) \leq C(\|\Upsilon\|_{C^1}+\|\Upsilon\|_{C^2}^2) \| f \|_\varphi^2$ and $I$ is the identity matrix.
\end{prop}

Henceforth we use the signed distance function as our defining function, so $\rho=\tilde\delta$.  When $\bd\Omega$ satisfies weak $Z(q)$, by Lemma \ref{lem:extension of Upsilon} we have
\begin{multline}\label{eqn:boundary term control}
 \sum_{I\in\I_{q-1}} \sum_{j,k=1}^n \big\la  \tilde\delta_{j\bar k} f_{jI}, f_{kI}  \big\ra
-\sum_{J\in \I_q}\sum_{j,k=1}^{n}\big\la \Upsilon^{\bar kj}\tilde\delta_{j\bar k}  f_J,f_J\big\ra \\
\geq (\mu_1 + \cdots + \mu_q)|f|^2 - |f|^2 \sum_{j,k=1}^n  \Upsilon^{\bar kj}\tilde\delta_{j\bar k} \geq0.
\end{multline}
Using linear algebra, we can show:
\begin{lem}\label{lem:G_j bounds E_j}
If $\Upsilon$ satisfies (i) of Lemma \ref{lem:extension of Upsilon}, then
\begin{multline*}
\sum_{J\in \I_q}\sum_{j,k=1}^{n}\left((I_{jk}-\Upsilon^{\bar kj})\frac{\partial f_J}{\partial\bar z_k},\frac{\partial f_J}{\partial\bar z_j}\right)_\varphi
 - 2\Rre\Bigg\{\sum_{J\in \I_q}\sum_{j,k,\ell=1}^{n} 				
   \left(\frac{\partial\Upsilon^{\bar kj}}{\partial z_j}(I_{k \ell}-\Upsilon^{\bar\ell k})\frac{\partial f_J}{\partial\bar z_\ell},f_J\right)_\varphi \Bigg\}\geq\\
    \frac 12 \sum_{J\in \I_q}\sum_{j,k=1}^{n}\left((I_{jk}-\Upsilon^{\bar kj})\frac{\partial f_J}{\partial\bar z_k},\frac{\partial f_J}{\partial\bar z_j}\right)_\varphi
 - O(\|f\|_\varphi^2)
\end{multline*}
and
\begin{multline*}
\sum_{J\in \I_q}\sum_{j,k=1}^{n}\left(\Upsilon^{\bar kj}L^\varphi_j  f_J,L^\varphi_k  f_J\right)_\varphi + 2\Rre\Bigg\{  \sum_{J\in \I_q}\sum_{j,k,\ell=1}^{n}\left(\frac{\partial\Upsilon^{\bar kj}}{\p\z_k}\Upsilon^{\bar j\ell}
  L^\varphi_\ell  f_J,f_J\right)_\varphi
  \Bigg\} \geq\\
   \frac 12 \sum_{J\in \I_q}\sum_{j,k=1}^{n}\left(\Upsilon^{\bar kj}L^\varphi_j  f_J,L^\varphi_k  f_J\right)_\varphi - O(\|f\|_\varphi^2)
\end{multline*}
where $O(\|f\|^2_\varphi)\leq C \|\Upsilon\|_{C^1}\|f\|_\varphi^2$ for some constant $C$.
\end{lem}

We are now ready to prove the basic estimate.
\begin{prop}\label{prop:basic estimate}
Let $\Omega$ have a connected boundary, a uniformly $C^m$ defining function for some $m\geq 2$, and satisfy weak $Z(q)$ for some $1\leq q\leq n-1$.  Suppose $\varphi$ satisfies $\ddbar\varphi=t\ddbar|z|^2$. Then for any
constant $C$, there exists a $t\in\R$  so that if $f\in L^2_{0,q}(\Omega,e^{-\varphi})\cap\Dom(\dbar)\cap\Dom(\dbars_\varphi)$, then
\[
\| \dbar f \|_\varphi^2 + \| \dbars_\varphi f\|_\varphi^2 \geq C \| f \|_\varphi^2.
\]
\end{prop}

\begin{proof} Observe that $q-\Tr(\Upsilon)$ is continuous, so there exists $\theta>0$ such that $q-\Tr(\Upsilon)\geq \theta>0$ or
$q-\Tr(\Upsilon) \leq -\theta <0$ for all $z\in\bar\Omega$. In the former case, we will take $t>0$ and in the latter case, we choose $t<0$.  Let $f_j$ and $\Omega_j$ be as in Lemma \ref{lem:bounded_density}.  Using
Proposition \ref{prop:basic identity}, Lemma \ref{lem:G_j bounds E_j},  and \eqref{eqn:boundary term control} it then follows that
\[
\| \dbar f_j\|_\varphi^2 + \| \dbars_\varphi f_j \|_\varphi^2
\geq  |t|\theta \sum_{J\in\I_q} \big((f_j)_J, (f_j)_J\big)_\varphi + O (\| f_j \|_\varphi^2),
\]
where $O(\|f_j\|_\varphi^2) \leq c(\|\Upsilon\|_{C^1}+ \| \Upsilon \|_{C^2}^2) \| f_j \|_\varphi^2$.
The result follows immediately by taking $|t|$ large enough and letting $j\rightarrow\infty$.
\end{proof}

Theorem \ref{thm:closed range for dbar} now follows from Proposition \ref{prop:basic estimate} by standard arguments. For example,
see \cite{HaRa11,Str10}.

\section{Closed Range in Unweighted Spaces}
\label{sec:closed_range_in_unweighted_spaces}

Although our primary goal in this paper is to prove estimates in weighted spaces, we briefly digress to consider closed range in unweighted $L^2$ spaces.  Note that we have stated Proposition \ref{prop:basic estimate} in sufficient generality to include unbounded domains with bounded weight functions (Herbig and McNeal \cite{HeMc16} have shown that a weight function with self-bounded gradient suffices for closed range).  Our model will be the strictly pseudoconvex domain $\Omega$ defined by the defining function $\rho(z)=\sum_{j=1}^n (\re z_j)^2-1$.  In this case, we can use the weight $\varphi=2t\sum_{j=1}^n (\re z_j)^2$.  We can easily check that this satisfies $\ddbar\varphi=t\ddbar|z|^2$.  Furthermore, $|\varphi|\leq 2t$ on $\Omega$, so we can use the arguments of H\"ormander \cite{Hor65} to obtain closed range in unweighted $L^2$ spaces (see also section 4.4 of \cite{ChSh01} for details on passing to unweighted estimates for $N_q$).

With this as our model, we can prove the following:
\begin{prop}Let $\Omega\subset \C^n$ be a domain with connected boundary and a uniformly $C^3$ defining function that
satisfies weak $Z(q)$ for some $1\leq q \leq n-1$.  Suppose there exists a unitary matrix $U_{jk}$ such that $\sum_{j=1}^n\left(\re\left(\sum_{k=1}^n U_{jk} z_k\right)\right)^2$ is bounded on $\Omega$.  Then
\begin{enumerate}\renewcommand{\labelenumi}{(\roman{enumi})}
\item The operator $\dbar: L^2_{0,\tilde q}(\Omega)\to L^2_{0,\tilde q+1}(\Omega)$ has closed range
for $\tilde q = q-1$ or $q$;
\item The operator $\dbars: L^2_{0,\tilde q+1}(\Omega)\to L^2_{0,\tilde q}(\Omega)$ has closed range
for $\tilde q = q-1$ or $q$;
\item The $\dbar$-Neumann Laplacian defined by $\Box_{q} = \dbar\dbars + \dbars\dbar$ has closed range on $L^2_{0,q}(\Omega)$;
\item The weighted $\dbar$-Neumann operator operator $N_{q}$  is continuous on $L^2_{0,q}(\Omega)$;
\item The canonical solution operators
for $\dbar$, $\dbars N_{q}:L^2_{0,q}(\Omega)\to L^2_{0,q-1}(\Omega)$
and\\ $N_{q}\dbars : L^2_{0,q+1}(\Omega)\to L^2_{0,q}(\Omega)$,
are continuous;
\item The canonical solution operators for $\dbars$, $\dbar N_{q}:L^2_{0,q}(\Omega)\to L^2_{0,q+1}(\Omega)$
and\\ $N_{q}\dbar : L^2_{0,q-1}(\Omega)\to L^2_{0,q}(\Omega)$, are continuous;
\item If $\tilde{q}=q$ or $q+1$ and $\alpha\in L^2_{0,\tilde{q}}(\Omega)$
so that  $\dbar \alpha =0$, then there exists $u\in L^2_{0,\tilde{q}-1}(\Omega)$ so that
\[
\dbar u = \alpha.
\]
\item $\ker(\Box_{q}) = \{0\}$. 
\end{enumerate}
\end{prop}

It might seem more desirable to restrict to domains on which we have estimates in unweighted $L^2$, as we have done in this section.  However, in a later paper we intend to prove estimates for the solution operator in weighted Sobolev spaces.  If we try to pass to Sobolev space estimates at this point, we will find that we have reached a dead end.  Suppose that $\Omega$ contains infinitely many disjoint balls $B_k$ of fixed radius $r$ (as is the case in our model domain defined by $\rho(z)=\sum_{j=1}^n (\re z_j)^2-1$, where we may consider the family of balls centered at $(0,\ldots,0,2ki)$ of radius $1$).  If we take any function $f\in C^\infty_0(B(0,r))$ and define $f_k(z)=f(z-c_k)$, where $c_k$ is the center of $B_k$, then we have a sequence $\{f_k\}$ that is uniformly bounded in $L^2$ with no convergent subsequence.  Hence, the Rellich Lemma is impossible, making any theory of Sobolev Spaces extremely problematic.

%
%
\section{Examples}\label{sec:examples}

We begin by correcting a proof from \cite{HaRa13} in order to show that we have a large class of domains with uniformly $C^m$ defining functions.

\begin{prop}
\label{prop:Cm_in_RPn}
  If $\tilde\Omega\subset\mathbb{RP}^n$ is a $C^m$ domain for some $m\geq 2$ and
  \[
    \Omega=\{x\in\mathbb{R}^n:[x_1:\ldots:x_n:1]\in\tilde\Omega\},
  \]
  then $\Omega$ is uniformly $C^m$.
\end{prop}

\begin{rem}
  This corrects the proof of Corollary 3.1 in \cite{HaRa13}.  The remark immediately following the proof of Corollary 3.1 does not follow from the corrected proof, and is probably not true without further assumptions (assuming that $\Omega$ and $\Omega^c$ are both asymptotically non-radial would suffice).
\end{rem}

\begin{proof}
  Let $\tilde\rho$ be a $C^m$ defining function for $\tilde\Omega$.  In homogeneous coordinates, $\tilde\rho:\mathbb{R}^{n+1}\backslash\{0\}\rightarrow\mathbb{R}$ satisfies $\tilde\rho(y)=\tilde\rho(\lambda y)$ for every $\lambda\in\mathbb{R}\backslash\{0\}$.  Then
  \[
    0=\frac{\partial}{\partial\lambda}\tilde\rho(\lambda y)|_{\lambda=1}=y\cdot\nabla\tilde\rho(y).
  \]
  We can construct a defining function $\rho$ for $\Omega$ by $\rho(x)=\tilde\rho([x_1:\ldots:x_n:1])$.  Thus,
  \[
    0=[x_1:\ldots:x_n:1]\cdot\nabla\tilde\rho([x_1:\ldots:x_n:1])=x\cdot\nabla\rho(x)+\frac{\partial\tilde\rho}{\partial y_{n+1}}([x_1:\ldots:x_n:1]),
  \]
  so we can write
  \begin{equation}
  \label{eq:x_dot_gradient_rho}
    x\cdot\nabla\rho(x)=-\tilde\rho_{n+1}([x_1:\ldots:x_n:1]).
  \end{equation}

  Since $\tilde\rho(y)=\tilde\rho(\lambda y)$, we have $\nabla^k\tilde\rho(y)=\lambda^k\nabla^k\tilde\rho(\lambda y)$ where $\nabla^k$ represents the vector of all $k$th order derivatives.  Hence $\nabla^k\tilde\rho(y)=\frac{\nabla^k\tilde\rho(y/|y|)}{|y|^k}$.  If $|\nabla^k\tilde\rho(y/|y|)|\leq C_k$ on $\partial\tilde\Omega$, we have $|\nabla^k\rho(x)|\leq\frac{C_k}{(|x|^2+1)^{k/2}}$.  On the other hand, \eqref{eq:x_dot_gradient_rho} implies
  \[
    |\tilde\rho_{n+1}([x_1:\ldots:x_n:1])|^2\leq|x|^2|\nabla\rho(x)|^2,
  \]
  so adding $|\nabla\rho(x)|^2$ to both sides gives us
  \[
    |\nabla\tilde\rho([x_1:\ldots:x_n:1])|^2\leq(|x|^2+1)|\nabla\rho(x)|^2.
  \]
  If $|\nabla\tilde\rho(y/|y|)|\geq C_0>0$ on $\partial\tilde\Omega$, then we have
  \[
    (|x|^2+1)|\nabla\rho(x)|^2\geq\frac{C_0^2}{|x|^2+1}.
  \]
  Hence, $|\nabla\rho(x)|\geq\frac{C_0}{|x|^2+1}$, so
  \[
    \frac{|\nabla^k\rho(x)|}{|\nabla\rho(x)|}\leq\frac{C_k}{C_0(|x|^2+1)^{k/2-1}}.
  \]
  The right-hand side is uniformly bounded provided that $k\geq 2$ (and this estimate is trivial when $k=1$).  By Theorem 1.3 in \cite{HaRa13}, $\Omega$ has a uniformly $C^m$ defining function.

\end{proof}

With this result, we can construct a simple example.  Fix $1\leq p\leq n-1$ and set
\[
  |z|^2_+=\sum_{j=1}^p|z_j|^2\text{ and }|z|^2_-=\sum_{j=p+1}^n |z_j|^2.
\]
Using Proposition \ref{prop:Cm_in_RPn}, the domain in $\mathbb{C}^n$ defined by $\rho(z)=|z|^2_+-|z|^2_-+1$ admits a uniformly $C^\infty$ defining function, since this is defined by a nondegenerate quadratic polynomial.  This domain is also radially non-asymptotic, but we omit the proof since this is not relevant for closed range.  This domain satisfies $Z(q)$ for any $q\neq n-p-1$.  When $q>n-p-1$, this is an example of a $Z(q)$ domain with at least $q+1$ negative eigenvalues at every point, which is impossible on bounded domains.  To confirm these claims, we define $\Upsilon^{\bar k j}=\delta_{jk}-\bar{z}_k z_j\left(|z|^2_-\right)^{-1}$ for $p+1\leq j,k\leq n$ and $\Upsilon^{\bar k j}=0$ otherwise.  Observe that $\rho(z)=0$ implies $|z|^2_-=|z|^2_++1\neq 0$, so $\Upsilon$ is well-defined whenever $\rho(z)=0$.  Then $\sum_{j=1}^n\Upsilon^{\bar k j}\rho_j=0$ and all eigenvalues of $\Upsilon$ are equal to either $0$ or $1$.  In fact, $\sum_{\ell=1}^n\Upsilon^{\bar k \ell}\rho_{\ell\bar j}=-\Upsilon^{\bar k j}$, so $\Upsilon$ defines a projection onto an eigenspace of the Levi-form with eigenvalues of $-1$.  Since the rank of $\Upsilon$ is $n-p-1$, the dimension of this eigenspace must be $n-p-1$.  If, instead, we project onto the vectors in the first $p$ coordinates that are orthogonal to $\partial\rho$, then we have a $p-1$ dimensional eigenspace corresponding to the eigenvector $1$.  To compute the remaining eigenvector, we consider the vector $v=\left(|z|^2_-z_1,\ldots,|z|^2_-z_p,|z|^2_+z_{p+1},\ldots,|z|^2_+z_n\right)$.  Then $\sum_{j=1}^n v^j\rho_j=0$ and
\[
  \sum_{j=1}^n v^j\rho_{j\bar k}=\begin{cases}
    |z|^2_-z_k,&1\leq k\leq p\\
    -|z|^2_+z_k,&p+1\leq k\leq n
  \end{cases}
  =\frac{|z|^2_--|z|^2_+}{|z|^2}v^k+\frac{2|z|^2_-|z|^2_+}{|z|^2}\overline{\rho_k}.
\]
Hence $v$ is also an eigenvector of the Levi-form with eigenvalue $\frac{|z|^2_--|z|^2_+}{|z|^2}$.  When $\rho(z)=0$ this is equal to $(1+2|z|^2_+)^{-1}$.  Thus, the sum of the $q$ smallest eigenvalues is equal to
\[
  \mu_1+\ldots+\mu_q=\begin{cases}
    -q&1\leq q\leq n-p-1\\
    -n+p+1+(1+2|z|^2_+)^{-1}+q-n+p&n-p\leq q\leq n
  \end{cases},
\]
so
\[
  \mu_1+\ldots+\mu_q-\sum_{j,k=1}^n\Upsilon^{\bar k j}\rho_{j\bar k}=\begin{cases}
    -q+n-p-1&1\leq q\leq n-p-1\\
    (1+2|z|^2_+)^{-1}+q-n+p&n-p\leq q\leq n
  \end{cases},
\]
and this is always nonnegative.  We have $\Tr\Upsilon\neq q$ provided that $n-p-1\neq q$, so $\Omega$ satisfies $Z(q)$ for all $q\neq n-p-1$.

For a more interesting example, we introduce the polynomial $P(z_1,z_2)=2x\abs{z_2}^2-x y^4$, where $z_1=x+iy$.  In \cite{HaRa15}, we show that the domain defined by $P(z_1,z_2)<\im z_3$ in $\mathbb{C}^3$ (formally) satisfies weak $Z(2)$ and requires the full flexibility of Definition \ref{defn:weak Z(q)}.  In the context of \cite{HaRa15}, this proof is only formal, because the domain is unbounded but the closed range results of \cite{HaRa15} only hold on bounded domains.  Hence, in that paper we must go through a lengthy argument to transform this into a bounded domain.  In the present paper, we are no longer restricted to bounded domains, but unfortunately this domain is not uniformly $C^2$.  Let $M$ denote the set where $y=0$, $z_2=0$, and $\im z_3=0$, and let $\rho=-\im z_3+P(z_1,z_2)$ be a defining function.  Then $\rho|_M=0$, $\partial\rho|_M=\frac{i}{2}dz_3$, and $\ddbar\rho|_M=2xdz_2\wedge d\bar{z}_2$, so by Lemma 2.3 in \cite{HaRa13} the domain defined by $\rho$ has no uniformly $C^2$ defining function.  Fortunately, this is easily corrected.

We would like to add $\frac{1}{3}(|z_1|^2+|z_2|^2)^3$ to $P(z_1,z_2)$ in order to improve the asymptotic behavior.  Unfortunately, to preserve the interesting properties of the example, we need to preserve the set where $x=z_2=0$.  To that end, we will truncate $\frac{1}{3}(|z_1|^2+|z_2|^2)^3$ to exclude all terms involving $y^4$ or $y^6$.  Unfortunately, this will cause uniform smoothness to fail on our critical set, so we subtract the pluriharmonic polynomial $\frac{1}{60}\re(z_1^6)$ in order to correct this without impacting the Levi-form.  In summary, our correction will take the form
\[
  Q(z_1,z_2)=\frac{1}{3}(x^2+|z_2|^2)^3+(x^2+|z_2|^2)^2 y^2-\frac{1}{60}x^6+\frac{1}{4}x^4y^2-\frac{1}{4}x^2y^4+\frac{1}{60}y^6,
\]
so our new example will be defined by:
\begin{equation}
\label{eq:interesting_example}
  \rho(z)=-\im z_3+P(z_1,z_2)+Q(z_1,z_2),
\end{equation}
For such a polynomial, we can use the argument of \cite{HaRa15} to show the following result.
\begin{prop}
\label{prop:interesting_example}
  There exists a domain in $\mathbb{C}^3$ with a uniformly $C^\infty$ defining function satisfying weak $Z(2)$ such that for some closed but unbounded set $K\subset\partial\Omega$, $\Upsilon$ in Definition \ref{defn:weak Z(q)} must have one eigenvalue equal to one on $K$ but on any neighborhood of $K$ there must be points where $\Upsilon$ has no eigenvalues equal to one.
\end{prop}

\begin{rem}
  As shown in \cite{HaRa15}, this implies that $\Omega$ is not $1$-pseudoconvex.  In the notation of the present paper, $1$-pseudoconvexity would mean that there exists $\Upsilon$ satisfying Definition \ref{defn:weak Z(q)} for weak $Z(2)$, but all eigenvalues of $\Upsilon$ must be one or zero.  This is impossible for this example, since the number of eigenvalues equal to one is not locally constant.
\end{rem}

\begin{rem}
  This domain is also asymptotically nonradial, but we omit the proof since this property is not relevant for closed range.
\end{rem}

\begin{rem}
  The conclusion of Proposition \ref{prop:interesting_example} is locally independent of the choice of hermitian metric, so $1$-pseudoconvexity will fail in any metric.  To see this, fix $p\in K$ with an open neighborhood $U$ and an arbitrary hermitian metric $g$ on $U$.  Suppose there exists $\Upsilon_g$ satisfying Definition \ref{defn:weak Z(q)} for weak $Z(2)$ on $U$.  Our Levi-form in this new metric can be computed in terms of the Euclidean Levi-form by $\mathcal{L}^g=\bar G^T\mathcal{L} G$ for some invertible matrix $G$.  Let $\lambda>0$ be a smooth function that is at least as large as the largest eigenvalue of the positive semi-definite matrix $G(I-\Upsilon_g)\bar G^T$, and define $\Upsilon=I-\lambda^{-1} G(I-\Upsilon_g)\bar G^T$.  Then all eigenvalues of $\Upsilon$ are bounded between zero and one,
  \[
    \Tr\mathcal{L}-\sum_{j,k=1}^2\Upsilon^{\bar k j}\mathcal{L}_{j\bar k}=\Tr((I-\Upsilon)\mathcal{L})=\lambda^{-1}\Tr((I-\Upsilon_g)\mathcal{L}^g)=\lambda^{-1}(\Tr\mathcal{L}^g-\sum_{j,k=1}^2\Upsilon^{\bar k j}_g\mathcal{L}^g_{j\bar k})\geq 0,
  \]
  and $\Tr\Upsilon<2$, so $\Upsilon$ satisfies Definition \ref{defn:weak Z(q)} on $U$.  Further, if $v$ is an eigenvector of $\Upsilon$ with eigenvalue of one, then $v_g=\bar G^T v$ is an eigenvalue of $\Upsilon_g$ with eigenvalue one, and vice-versa since $G$ is invertible.  Using Proposition \ref{prop:interesting_example}, $\Upsilon_g$ must have one eigenvalue equal to one on $K\cap U$, but there must also exist a point in $U$ where $\Upsilon_g$ has no eigenvalues equal to one.
\end{rem}

\begin{proof}
  We will first show that there exists $C>0$ such that $|\nabla\rho|\geq C(1+(|z_1|^2+|z_2|^2)^{5/2})$ when $\rho(z)=0$.  Suppose there exists a sequence $\{z^j\}$ in $\partial\Omega$ such that $|\nabla\rho(z^j)|<\frac{1}{j}(1+(|z_1^j|^2+|z_2^j|^2)^{5/2})$.  Since $\frac{\partial\rho}{\partial z_3}=-\frac{1}{2i}$, we have $|\nabla\rho(z^j)|\geq 1$, so we must have $(z_1^j,z_2^j)\rightarrow\infty$.  Let $(v_1,v_2)$ be a limit point of $\left\{\frac{(z_1^j,z_2^j)}{\sqrt{|z_1^j|^2+|z_2^j|^2}}\right\}$, and note that this must be a unit-length vector.  Since
  \[
    \frac{\partial\rho}{\partial z_2}=\bar{z}_2(2x+(x^2+|z_2|^2)^2+2(x^2+|z_2|^2)y^2),
  \]
  homogeneity and degree considerations establish that $\bar{v}_2(((\re v_1)^2+|v_2|^2)^2+2((\re v_1)^2+|v_2|^2)(\im v_1)^2)=0$, which is only possible if $v_2=0$.  Further differentiating, we have
  \[
    \frac{\partial\rho}{\partial y}=-4xy^3+y\left(2|z_2|^4+4x^2|z_2|^2+\frac{1}{10}(y^2-5x^2)^2\right),
  \]
  so $\frac{1}{10}(\im v_1)((\im v_1)^2-5(\re v_1)^2)^2=0$, which implies either $\im v_1=0$ or $\im v_1=\pm\sqrt{5}\re v_1$.  Finally, we check
  \[
    \frac{\partial\rho}{\partial x}=2|z_2|^2-y^4+x\left(\frac{19}{10}x^4+4x^2|z_2|^2+2|z_2|^4+4|z_2|^2y^2+5x^2y^2-\frac{1}{2}y^4\right),
  \]
  which implies $(\re v_1)\left(\frac{19}{10}(\re v_1)^4+5(\re v_1)^2(\im v_1)^2-\frac{1}{2}(\im v_1)^4\right)=0$.  Substituting either $\im v_1=0$ or $\im v_1=\pm\sqrt{5}\re v_1$ immediately implies $\re v_1=0$.  We conclude that $v=0$, contradicting the assumption that $v$ has unit length.  Hence, there must exist $C>0$ such that $|\nabla\rho|\geq C(1+(|z_1|^2+|z_2|^2)^{5/2})$ when $\rho(z)=0$.

  Now, we observe that for any $k\geq 1$, $|\nabla^k\rho|\leq O(1+(|z_1|^2+|z_2|^2)^{3-k/2})$, so $\frac{|\nabla^k\rho|}{|\nabla\rho|}\leq O(1+(|z_1|^2+|z_2|^2)^{1/2-k/2})$.  By the Main Theorem of \cite{HaRa13}, the domain defined by \eqref{eq:interesting_example} has a uniformly $C^\infty$ defining function.

  To understand the remaining properties, we differentiate \eqref{eq:interesting_example} (recall that the last four terms of $Q(z_1,z_2)$ sum to a pluriharmonic function, and can be neglected) to find
  \begin{multline*}
    \ddbar\rho=(-3xy^2+(x^2+|z_2|^2)^2+2x^2(x^2+|z_2|^2)+(x^2+|z_2|^2)y^2+2x^2y^2) dz_1\wedge d\bar{z}_1\\
    +z_2(1+2x(x^2+y^2+|z_2|^2)-2iy(x^2+|z_2|^2))dz_1\wedge d\bar{z}_2\\
    +\bar{z}_2(1+2x(x^2+y^2+|z_2|^2)+2iy(x^2+|z_2|^2))dz_2\wedge d\bar{z}_1\\
    +(2x+(x^2+|z_2|^2)^2+2(x^2+|z_2|^2)|z_2|^2+2(x^2+|z_2|^2)y^2+2|z_2|^2y^2)dz_2\wedge d\bar{z}_2.
  \end{multline*}

  The only component of $\partial\rho$ that we have yet to compute is $\frac{\partial\rho}{\partial z_3}=\frac{i}{2}$.  If we normalize our metric so that $|dz_j|^2=2$, then
  \[
    |d\rho(z)|^2=2|\partial\rho(z)|^2=4\sum_{j=1}^3|\rho_j(z)|^2=1+4|\rho_1(z)|^2+4|\rho_2(z)|^2.
  \]
  In what follows, we will frequently make use of the identity $\frac{4|\rho_1|^2+4|\rho_2|^2}{|d\rho|+1}=|d\rho|-1$.  Using this, we can construct orthonormal coordinates
  \begin{align*}
    u_1&=|d\rho|^{-1}\left(1+\frac{4|\rho_2|^2}{|d\rho|+1},-\frac{4\rho_1\overline{\rho_2}}{|d\rho|+1},2i\rho_1\right),\\
    u_2&=|d\rho|^{-1}\left(-\frac{4\overline{\rho_1}\rho_2}{|d\rho|+1},1+\frac{4|\rho_1|^2}{|d\rho|+1},2i\rho_2\right),\\
    u_3&=|d\rho|^{-1}(2\overline{\rho_1},2\overline{\rho_2},-i).
  \end{align*}
  Since $u_3$ is the complex normal vector to a uniformly $C^\infty$ domain, each component must be a uniformly $C^\infty$ function.  Note that $u_1$ and $u_2$ can be expressed entirely in terms of the coefficients of $u_3$.  For example, the first component of $u_1$ can be written $i\frac{-i}{|d\rho|}+\frac{\abs{2\overline{\rho_2}/|d\rho|}^2}{1+i(-i/|d\rho|)}$.  In each case, the denominator of $1+|d\rho|^{-1}$ is uniformly bounded away from zero, so each coefficient must also be uniformly $C^\infty$.  Thus, we can work in these coordinates without sacrificing any conclusions about uniform smoothness.

  As a notational convenience, we write $u_j=(u_j^1,u_j^2,u_j^3)$.  Using our orthonormal coordinates, we can compute the Levi-form as follows:
  \[
    \mathcal{L}_{j\bar k}=\sum_{\ell,m=1}^3u_j^\ell\rho_{\ell\bar m}\overline{u_k^m}.
  \]
  Since all of the terms with $\ell=3$ or $m=3$ vanish, we can rewrite this as a product of $2\times 2$ matrices, so that $\mathcal{L}=U\begin{pmatrix}\rho_{1\bar 1}&\rho_{1\bar 2}\\\rho_{2\bar 1}&\rho_{2\bar 2}\end{pmatrix}\bar U^T$, where $U=\begin{pmatrix}u_1^1&u_1^2\\u_2^1&u_2^2\end{pmatrix}$.  Note that $U$ is also hermitian, so the conjugate transpose is not necessary in this expression.

  Using this, we deduce that the Levi-form is trivial only when $\ddbar\rho=0$.  When $x\geq 0$, $\rho_{2\bar 2}=0$ if and only if $x=0$ and $z_2=0$.  When $x\leq 0$, $\rho_{1\bar 1}=0$ if and only if $x=0$ and $z_2=0$.  Hence, regardless of the sign of $x$, we can only have $\ddbar\rho=0$ if $x=0$ and $z_2=0$. We easily check that the converse is also true, so we define
  \[
    K_0=\{z\in\partial\Omega:x=0\text{ and }z_2=0\},
  \]
  which is the set where the Levi-form is trivial.

  To prove that $\Omega$ satisfies weak $Z(2)$, we will break the problem down into three steps.  First, we will show that $\Omega$ is pseudoconvex on some set that is bounded a uniform distance away from $K_0$, so we can use $\Upsilon=0$ on this set.  Next, we will construct $\Upsilon$ with rank $1$ on a uniform neighborhood of $K_0$ provided that $|y|$ is bounded away from zero.  Finally, we will use the construction from \cite{HaRa13} to cover a uniform neighborhood of $K_0$ when $|y|$ is also uniformly bounded.  We will see that by checking these three cases in this order, we can guarantee that the three sets overlap in a uniform way, so it will be possible to patch each $\Upsilon$ together and obtain a global $\Upsilon$.

  We begin by considering a set where $\Omega$ is pseudoconvex.  Because the Levi-form is defined by conjugating a submatrix of the complex hessian of $\rho$ by $U$, we can see that $\Omega$ is pseudoconvex precisely at points where $\begin{pmatrix}\rho_{1\bar 1}&\rho_{1\bar 2}\\\rho_{2\bar 1}&\rho_{2\bar 2}\end{pmatrix}$ is positive semi-definite.  Suppose $\sqrt{x^2+|z_2|^2}\geq R_0$ for some $R_0>0$ to be determined later.  Then $-3xy^2\geq -3R_0^{-1}(x^2+|z_2|^2)y^2$ and $2x\geq-2R_0^{-3}(x^2+|z_2|^2)^2$, so
  \[
    \rho_{1\bar 1}\geq(x^2+|z_2|^2)^2+2x^2(x^2+|z_2|^2)+(1-3R_0^{-1})(x^2+|z_2|^2)y^2+2x^2y^2,
  \]
  and
  \[
    \rho_{2\bar 2}\geq(1-2R_0^{-3})(x^2+|z_2|^2)^2+2(x^2+|z_2|^2)|z_2|^2+2(x^2+|z_2|^2)y^2+2|z_2|^2y^2.
  \]
  These are both positive provided that $R_0\geq 3$.  Similarly, since $1\leq R_0^{-3}(x^2+|z_2|^2)^{3/2}$,
  \[
    (1+2x(x^2+y^2+|z_2|^2))^2\leq (R_0^{-6}+R_0^{-3})(x^2+|z_2|^2)^3+4x^2(1+R_0^{-3})(x^2+y^2+|z_2|^2)^2,
  \]
  where we have used $(s+t)^2\leq (1+R_0^3)s^2+(1+R_0^{-3})t^2$.  Hence
  \begin{multline*}
    |\rho_{1\bar 2}|^2\leq|z_2|^2\left(4x^2(x^2+y^2+|z_2|^2)^2+4y^2(x^2+|z_2|^2)^2\right)\\
    +O(R_0^{-3}(x^2+|z_2|^2)^2(x^2+y^2+|z_2|^2)^2).
  \end{multline*}
  Using our estimates on $\rho_{j\bar k}$, we expand $\det (\rho_{j\bar k})$ as a polynomial  in $y^2$ and observe that
  \[
    \rho_{1\bar 1}\rho_{2\bar 2}-|\rho_{1\bar 2}|^2\geq a+by^2+cy^4-O(R_0^{-1}(x^2+|z_2|^2)^2(x^2+y^2+|z_2|^2)^2),
  \]
  where $a$, $b$, and $c$ are all polynomials in $x$ and $|z_2|^2$.  We compute
  \begin{align*}
    a&=3(x^2+|z_2|^2)^4, \\
    b&=3(x^2+|z_2|^2)^2(3x^2+|z_2|^2),\\
 	\intertext{and}
    c&=2(x^2+|z_2|^2)(3x^2+2|z_2|^2).
  \end{align*}
  Hence, we can choose $R_0$ sufficiently large so that $\rho_{1\bar 1}\rho_{2\bar 2}-|\rho_{1\bar 2}|^2>0$.  This implies that $\Omega$ is strictly pseudoconvex when $\sqrt{x^2+|z_2|^2}\geq R_0$.  On this set, we can use the trivial value $\Upsilon_0=0$.

  For $R_1>R_0$ and $Y_1>0$, we will now consider the set where $\sqrt{x^2+|z_2|^2}\leq R_1$ and $|y|\geq Y_1$.  On this set, define $\Upsilon_1=\frac{4\lambda}{|d\rho|^2}\begin{pmatrix}|\rho_2|^2&-\rho_1\overline{\rho_2}\\-\overline{\rho_1}\rho_2&|\rho_1|^2\end{pmatrix}$, where $\lambda=1-\frac{150y^2-100}{y^{10}+100y^8}$.  Note that we can choose $Y_1$ sufficiently large so that $100 y^{-8}<1-\lambda<200 y^{-8}$.  This guarantees $0<\lambda<1$ if $Y_1$ is sufficiently large.  On such a set, $\lambda$ and its derivatives are all uniformly bounded.  The same is also true for $\frac{\rho_1}{|d\rho|}$ and $\frac{\rho_2}{|d\rho|}$, so $\Upsilon_1$ must be uniformly smooth.  We have $\det\Upsilon_1=0$ and $\Tr\Upsilon_1=\frac{4\lambda(|\rho_1|^2+|\rho_2|^2)}{|d\rho|^2}$, so the eigenvalues of $\Upsilon_1$ are equal to $0$ and $\Tr\Upsilon_1$.  Since we have $0<\Tr\Upsilon_1<1$, it remains to consider $\Tr\mathcal{L}-\sum_{j,k=1}^2\Upsilon_1^{\bar k j}\mathcal{L}_{j\bar k}$.  To compute the trace of the Levi-form, we use
  \[
    \Tr\mathcal{L}=\sum_{j,\ell,m=1}^2u_j^\ell\rho_{\ell\bar m}\overline{u_j^m}.
  \]
  Thus, computing the trace requires computing $\bar U^T U=\begin{pmatrix}|u_1^1|^2+|u_2^1|^2&u_1^2\overline{u_1^1}+u_2^2\overline{u_2^1}\\u_1^1\overline{u_1^2}+u_2^1\overline{u_2^2}&|u_2^2|^2+|u_1^2|^2\end{pmatrix}$.
  Since the $3\times 3$ matrix $(u_j^k)$ is necessarily a unitary matrix, we have
  \begin{equation}
  \label{eq:U_squared}
    \bar U^T U=\begin{pmatrix}1-|u_3^1|^2&-u_3^2\overline{u_3^1}\\-u_3^1\overline{u_3^2}&1-|u_3^2|^2\end{pmatrix}=|d\rho|^{-2}\begin{pmatrix}1+4|\rho_2|^2&-4\rho_1\overline{\rho_2}\\-4\overline{\rho_1}\rho_2&1+4|\rho_1|^2\end{pmatrix}
  \end{equation}
  Therefore
  \[
    \Tr\mathcal{L}=|d\rho|^{-2}\left((1+4|\rho_2|^2)\rho_{1\bar 1}-8\re(\overline{\rho_1}\rho_2\rho_{1\bar 2})+(1+4|\rho_1|^2)\rho_{2\bar 2}\right).
  \]
  Since $\begin{pmatrix}\overline{\rho_2}\\-\overline{\rho_1}\end{pmatrix}$ is an eigenvector of $U$ with eigenvalue $1$, we have $\bar U^T\Upsilon_1 U=\Upsilon_1$, so
  \[
    \sum_{j,k=1}^2\Upsilon_1^{\bar k j}\mathcal{L}_{j\bar k}=4\lambda|d\rho|^{-2}(|\rho_2|^2\rho_{1\bar 1}-2\re(\overline{\rho_1}\rho_2\rho_{1\bar 2})+|\rho_1|^2\rho_{2\bar 2}).
  \]
  Combining these computations, we have
  \begin{multline*}
    \Tr\mathcal{L}-\sum_{j,k=1}^2\Upsilon_1^{\bar k j}\mathcal{L}_{j\bar k}\\
    =|d\rho|^{-2}\left((1+4(1-\lambda)|\rho_2|^2)\rho_{1\bar 1}-8(1-\lambda)\re(\overline{\rho_1}\rho_2\rho_{1\bar 2})+(1+4(1-\lambda)|\rho_1|^2)\rho_{2\bar 2}\right).
  \end{multline*}
  To estimate this quantity, it will be helpful to observe that whenever $m_2>m_1$, we have $|y|^{m_1}\leq Y_1^{m_1-m_2}|y|^{m_2}\leq O(|y|^{m_2})$ and $(x^2+|z_2|^2)^{m_2/2}\leq R_1^{m_2-m_1}(x^2+|z_2|^2)^{m_1/2}\leq O((x^2+|z_2|^2)^{m_1/2})$, since $R_1$ is fixed and $Y_1$ can be taken arbitrarily large.  With this in mind, we estimate
  \[
    \abs{\rho_1+\frac{1}{2}y^4+\frac{1}{4}xy^4+\frac{i}{20}y^5}\leq O((x^2+|z_2|^2)^{1/2}y^3).
  \]
  From this, we immediately obtain $|\rho_1|\leq O(y^5)$ and
  \[
    |\rho_1|^2\geq\frac{y^8}{4}+\frac{y^{10}}{400}-O((x^2+|z_2|^2)^{1/2}y^8).
  \]
  We will also need $|\rho_2|\leq O((x^2+|z_2|^2)y^2)$ and $|\rho_{1\bar 2}|\leq O((x^2+|z_2|^2)^{1/2}y^2)$
  From these, we have
  \[
    \re(\overline{\rho_1}\rho_2\rho_{1\bar 2})\leq O((x^2+|z_2|^2)^{3/2}y^9).
  \]
  Since most of the terms in $\rho_{1\bar 1}$ and $\rho_{2\bar 2}$ are positive, we won't need error terms to estimate $\rho_{1\bar 1}\geq-3xy^2$ and $\rho_{2\bar 2}\geq 2x+(2x^2+4|z_2|^2)y^2$.  Since $1+4(1-\lambda)|\rho_j|^2\geq 0$ for $j\in\{1,2\}$, we can first use these lower bounds on $\rho_{1\bar 1}$ and $\rho_{2\bar 2}$ before estimating the remaining error terms.  Since $|1-\lambda|\leq 200 y^{-8}$, we have
  \begin{multline*}
    \Tr\mathcal{L}-\sum_{j,k=1}^2\Upsilon_1^{\bar k j}\mathcal{L}_{j\bar k}
    \geq|d\rho|^{-2}\left(-3xy^2+(1+4(1-\lambda)|\rho_1|^2)(2x+(2x^2+4|z_2|^2)y^2)\right)\\
    -O(|d\rho|^{-2}(x^2+|z_2|^2)y^2).
  \end{multline*}
  Substituting our lower bound for $|\rho_1|^2$, we are left with
  \begin{multline*}
    \Tr\mathcal{L}-\sum_{j,k=1}^2\Upsilon_1^{\bar k j}\mathcal{L}_{j\bar k}
    \geq|d\rho|^{-2}\left(-3xy^2+2x\right)\\
    +4(1-\lambda)|d\rho|^{-2}\left(2x\left(\frac{y^8}{4}+\frac{y^{10}}{400}\right)+(2x^2+4|z_2|^2)\frac{y^{12}}{400}\right)
    -O(|d\rho|^{-2}(x^2+|z_2|^2)y^2).
  \end{multline*}
  The value of $\lambda$ has been chosen so that all terms that are linear in $x$ will cancel, leaving us with
  \[
    \Tr\mathcal{L}-\sum_{j,k=1}^2\Upsilon_1^{\bar k j}\mathcal{L}_{j\bar k}
    \geq 4(1-\lambda)|d\rho|^{-2}\left((2x^2+4|z_2|^2)\frac{y^{12}}{400}\right)
    -O(|d\rho|^{-2}(x^2+|z_2|^2)y^2).
  \]
  Since $1-\lambda\geq 100 y^{-8}$, we can choose $Y_1$ sufficiently large so that this quantity is greater than or equal to zero.

  For our final region, we choose $R_2=R_1$ and $Y_2>Y_1$.  When $\sqrt{x^2+|z_2|^2}\leq R_2$ and $|y|\leq Y_2$ we set
  \[
    \Upsilon_2=I-\big(2(1+4|\rho_1|^2)+3y^2(1+4|\rho_2|^2)\big)^{-1} \left(\bar U^T\right)^{-1}\begin{pmatrix}2&0\\0&3y^2\end{pmatrix}U^{-1}.
  \]
  Since
  \[
    \det U=|d\rho|^{-2}\left(1+\frac{4|\rho_1|^2}{|d\rho|+1}+\frac{4|\rho_2|^2}{|d\rho|+1}\right)=|d\rho|^{-1},
  \]
  we can invert \eqref{eq:U_squared} to obtain
  \[
    U^{-1}\left(\bar U^T\right)^{-1}=\begin{pmatrix}1+4|\rho_1|^2&4\rho_1\overline{\rho_2}\\4\overline{\rho_1}\rho_2&1+4|\rho_2|^2\end{pmatrix},
  \]
  so we can compute
  \[
    \Tr\Upsilon_2=2-(2(1+4|\rho_1|^2)+3y^2(1+4|\rho_2|^2))^{-1}(2(1+4|\rho_1|^2)+3y^2(1+4|\rho_2|^2))=1.
  \]
  Since $I-\Upsilon_2$ is positive semi-definite, each eigenvalue of $\Upsilon_2$ must be at most one.  Two eigenvalues less than or equal to one can only add to one if both eigenvalues are also greater than or equal to zero, so each eigenvalue of $\Upsilon_2$ lies on the interval $[0,1]$.  We check
  \[
    \Tr\mathcal{L}-\sum_{j,k=1}^2\Upsilon_2^{\bar k j}\mathcal{L}_{j\bar k}=(2(1+4|\rho_1|^2)+3y^2(1+4|\rho_2|^2))^{-1}(2\rho_{1\bar 1}+3y^2\rho_{2\bar 2})\geq 0,
  \]
  since the only potentially negative terms in $\rho_{1\bar 1}$ and $\rho_{2\bar 2}$ are $-3xy^2$ and $2x$, respectively.  Since the denominator is bounded below by 2 and $\Upsilon_2$ is only defined on a compact set, $\Upsilon_2$ is uniformly smooth.

  We are now ready to assemble our $\Upsilon$.  Let $\chi\in C^\infty(\mathbb{R})$ be a nondecreasing function satisfying $\chi(t)=0$ when $t\leq 0$, $\chi(t)=1$ when $t\geq 1$, and $1-\chi(t)=\chi(1-t)$.  Define
  \[
    \Upsilon=\chi\left(\frac{y^2-Y_1^2}{Y_2^2-Y_1^2}\right)\chi\left(\frac{R_1^2-x^2-|z_2|^2}{R_1^2-R_0^2}\right)\Upsilon_1+\chi\left(\frac{Y_2^2-y^2}{Y_2^2-Y_1^2}\right)\chi\left(\frac{R_2^2-x^2-|z_2|^2}{R_2^2-R_0^2}\right)\Upsilon_2.
  \]
  This will satisfy all of the necessary properties of Definition \ref{defn:weak Z(q)}, so $\Omega$ satisfies weak $Z(2)$.

  For our negative result, we will examine the Levi-form in a neighborhood of $K_0$.  Fix $y$, and assume that $x^2+|z_2|^2<R$ for some $R>0$.  Since $|\rho_2(z)|\leq O(x^2+|z_2|^2)$, it follows that
  \[
    U=\frac{1}{\sqrt{1+4|\rho_1|^2}}\begin{pmatrix}1&0\\0&1+\frac{4|\rho_1|^2}{\sqrt{1+4|\rho_1|^2}+1}\end{pmatrix}+O(x^2+|z_2|^2).
  \]
  Note that $1+\frac{4|\rho_1|^2}{\sqrt{1+4|\rho_1|^2}+1}=\sqrt{1+4|\rho_1|^2}$, so we have
  \[
    \mathcal{L}=\frac{1}{1+4|\rho_1|^2}\begin{pmatrix}-3xy^2&z_2\sqrt{1+4|\rho_1|^2}\\\bar{z}_2\sqrt{1+4|\rho_1|^2}&2x(1+4|\rho_1|^2)\end{pmatrix}+O(x^2+|z_2|^2).
  \]

  Suppose that $\Upsilon$ satisfies the requirements of Definition \ref{defn:weak Z(q)} for weak $Z(2)$.  When $x=0$, we have
  \[
    \Tr\mathcal{L}-\sum_{j,k=1}^2\Upsilon^{\bar k j}\mathcal{L}_{j\bar k}=-2\re\left(\Upsilon^{\bar 2 1}\frac{z_2}{\sqrt{1+4|\rho_1|^2}}\right)+O(|z_2|^2)
  \]
  This quantity must be nonnegative for $z_2$ near zero and equal to zero when $z_2=0$, so the linear terms in $z_2$ must vanish.  Hence $\Upsilon^{\bar 2 1}=0$ on $K_0$.  On the other hand, when $z_2=0$ we have
  \[
    \Tr\mathcal{L}-\sum_{j,k=1}^2\Upsilon^{\bar k j}\mathcal{L}_{j\bar k}=\frac{-3xy^2}{1+4|\rho_1|^2}(1-\Upsilon^{\bar 1 1})+2x(1-\Upsilon^{\bar 2 2}) + O(x^2).
  \]
  As before, this must be nonnegative for $x$ near zero and equal to zero when $x=0$, so we must have $\frac{-3y^2}{1+4|\rho_1|^2}(1-\Upsilon^{\bar 1 1})+2(1-\Upsilon^{\bar 2 2})=0$ on $K_0$.

  On $K_0$, we now know that $\Upsilon$ is diagonal, so the eigenvalues must be $\Upsilon^{\bar 1 1}$ and $\Upsilon^{\bar 2 2}$.  When $y=0$, we must have $\Upsilon^{\bar 2 2}=1$, so since $\Tr\Upsilon\neq 2$ by assumption, we must have $\Upsilon^{\bar 1 1}\neq 1$.  When $y\neq 0$, we have $\Upsilon^{\bar 1 1}=1$ if and only if $\Upsilon^{\bar 2 2}=1$, so $\Tr\Upsilon\neq 2$ implies that neither eigenvalue can equal 1.  This proves our negative result, since $\Upsilon$ has exactly one eigenvalue equal to 1 on the set
  \[
    K=\{z\in\partial\Omega:z_1=z_2=0\},
  \]
  but no eigenvalues equal to 1 on the set $K_0\backslash K$.

\end{proof}

\bibliographystyle{alpha}
\bibliography{mybib}

\end{document}